\documentclass[12pt,twoside]{amsart}
\usepackage{amsmath}
\usepackage{amsthm}
\usepackage{amsfonts}
\usepackage{amssymb}
\usepackage{latexsym}
\usepackage{mathrsfs}
\usepackage{amsmath}
\usepackage{amsthm}
\usepackage{amsfonts}
\usepackage{amssymb}
\usepackage{latexsym}
\usepackage{geometry}
\usepackage{dsfont}
\usepackage[dvips]{graphicx}
\usepackage{color}
\usepackage[all]{xy}

\date{}
\allowdisplaybreaks[4] \footskip=15pt
\renewcommand{\uppercasenonmath}[1]{}

\usepackage{graphicx,amssymb}
\usepackage[all]{xy}
\usepackage{amsmath}

\allowdisplaybreaks
\usepackage{amsthm}
\usepackage{color}

\theoremstyle{plain}

\theoremstyle{plain}
\newtheorem{theorem}{Theorem}[section]
\newtheorem{proposition}[theorem]{Proposition}
\newtheorem{lemma}[theorem]{Lemma}
\newtheorem{corollary}[theorem]{Corollary}
\newtheorem{example}[theorem]{Example}
\newtheorem*{open question}{Open Question}
\newtheorem{definition}[theorem]{Definition}

\theoremstyle{definition}
\newtheorem*{acknowledgement}{Acknowledgement}

\theoremstyle{remark}
\newtheorem{remark}[theorem]{Remark}
\newcommand{\ra}{\rightarrow}

\newcommand{\Tor}{\mbox{\rm Tor}}

\newcommand{\coker}{\mbox{\rm coker}}
\newcommand{\Prufer}{Pr\"{u}fer}

\newcommand{\Q}{\mathcal{Q}}

\newcommand{\prodi}{\prod_{i\in \Lambda}}

\def\ra{\rightarrow}
\def\GV{{\rm GV}}
\def\tor{{\rm tor}}
\def\Hom{{\rm Hom}}
\def\Ext{{\rm Ext}}
\def\Tor{{\rm Tor}}

\def\fPD{{\rm fPD}}

\def\Ker{{\rm Ker}}
\def\coker{{\rm coker}}
\def\Im{{\rm Im}}
\def\Cok{{\rm Cok}}

\def\Nil{{\rm Nil}}

\def\GV{{\rm GV}}
\def\Max{{\rm Max}}
\def\DW{{\rm DW}}
\def\WQ{{\rm WQ}}
\def\T{{\rm T}}
\def\E{{\rm E}}
\def\DQ{{\rm DQ}}
\def\Spec{{\rm Spec}}
\newcommand{\m}{\frak{m}}
\newcommand{\p}{\frak{p}}
\def\Min{{\rm Min}}

\begin{document}
\begin{center}
{\large  \bf  On $\tau_q$-flatness and $\tau_q$-coherence}

\vspace{0.5cm}  Xiaolei Zhang$^{a}$, Wei Qi$^{a}$ \\

{\footnotesize a.\ School of Mathematics and Statistics, Shandong University of Technology, Zibo 255000, China\\

E-mail: qwrghj@126.com\\}
\end{center}

\bigskip
\centerline { \bf  Abstract}
\bigskip
\leftskip10truemm \rightskip10truemm \noindent

In this paper, we first introduce and study the notions of $\tau_q$-flat modules. And then study the rings over which all modules are $\tau_q$-flat (which is called $\tau_q$-VN regular rings). In particular, we show that a ring $R$ is a $\tau_q$-VN regular ring if and only if $\T(R[x])$ is a von Neumann regular ring, if and only if $R$ is a reduced ring with $\Min(R)$ compact. Finally, we obtain the Chase theorem for  $\tau_q$-coheret rings: a ring $R$ is $\tau_q$-coherent if and only if any direct product of $R$ is $\tau_q$-flat if and only if any direct product of flat $R$-modules is $\tau_q$-flat. Some examples are provided to compare with the known conceptions.
\vbox to 0.3cm{}\\
{\it Key Words:} small finitistic dimension; $\tau_q$-flat module; $\tau_q$-VN regular ring; $\tau_q$-coherent ring.\\
{\it 2010 Mathematics Subject Classification:} 13C11; 13B30; 13D05.

\leftskip0truemm \rightskip0truemm
\bigskip

\section{Introduction}
Throughout this paper, we always assume $R$ is a commutative ring with identity. For a ring $R$, we denote by $\T(R)$ the total quotient ring of $R$, $\Min(R)$ the set of all minimal prime ideals of $R$, $\Nil(R)$ the nil radical of $R$ and $\fPD(R)$ the small finitistic dimension of $R$ (see \cite{z-fpd}).

The notion of $w$-operations was first introduced by Wang et al. \cite{fm97} to build a connection of some non-coherent domains (e.g. resp., Krull domains, PvMDs and Strong Mori domains) with classical domains (e.g. resp., Dedekind domains, \Prufer\ domains and Noetherian domains). In 2011, Yin et al. \cite{ywzc11} extended $w$-operations to commutative rings with zero-divisors. This makes it possible to study modules over commutative rings in terms of star operations. In 2015,  Kim et al.  \cite{KW14} extended the classical notion of flat modules to that of $w$-flat modules. And then Wang et al. \cite{fk14} showed that a ring $R$ is a von Neumann regular ring if and only if every $R$-module is $w$-flat. Later, Wang et al. \cite{WQ15} constructed the $w$-weak global dimensions of commutative rings, and proved that PvMDs are exactly integral domains with $w$-weak global dimensions at most $1$. It is well-known that \Prufer\ domains are coherent domains. Correspondingly,  PvMDs should be some $w$-analogue of coherent ring. So early in 2010,  Wang \cite{fk10} introduced the notion of $w$-coherent rings and then  gave some examples by the Milnor square in \cite{fk12}. For module-theoretic study of $w$-coherent rings, Wang and the author et al. \cite{WQ20,zw-c-wch-2,zxl20} characterized $w$-coherent rings using $w$-flat modules until recently. Actually, they gave the Chase theorem for $w$-coherent rings.

Recently, Zhou et al. \cite{ZDC20} introduced the notion of $q$-operations over commutative rings utilizing finitely generated semi-regular ideals.  $q$-operations are semi-star operations which are weaker than $w$-operations.  The authors in \cite{ZDC20} also proposed $\tau_q$-Noetherian rings (i.e. a ring in which any ideal is $\tau_q$-finitely generated) and study them via module-theoretic point of view, such as $\tau_q$-analogue of the Hilbert basis theorem, Krull's principal ideal theorem, Cartan-Eilenberg-Bass theorem and Krull intersection theorem. The main motivation of this paper is to introduce the notions of $\tau_q$-versions of flat modules and coherent rings, and then give some module-theoretic studies of these notions.

As our work involves $q$-operations, we give a brief introduction on them. For more details, refer to  \cite{ZDC20}. Recall that an ideal $I$ of $R$ is said to be \emph{dense} if $(0:_RI):=\{r\in R\mid Ir=0\}$, and \emph{semi-regular} if there exists a finitely generated dense sub-ideal of $I$. The set of all finitely generated semi-regular ideals of $R$ is denoted by $\Q$. Let $M$ be an $R$-module. Denote by
 \begin{center}
{\rm $\tor_{\Q}(M):=\{x\in M|Ix=0$, for some $I\in \Q \}.$}
\end{center}
Recall from \cite{wzcc20} that an $R$-module $M$ is said to be \emph{$\Q$-torsion} (resp., \emph{$\Q$-torsion-free}) if $\tor_{\Q}(M)=M$ (resp., $\tor_{\Q}(M)=0$). The class of $\Q$-torsion modules is closed under submodules, quotients and direct sums, and the class of $\Q$-torsion-free modules is closed under submodules and direct products. A $\Q$-torsion-free module $M$ is called a \emph{Lucas module} if $\Ext_R^1(R/J,M)=0$ for any $J\in \Q$, and the \emph{Lucas envelope} of $M$ is given by
\begin{center}
{\rm $M_q:=\{x\in \E_R(M)|Ix\subseteq M$, for some $I\in \Q \},$}
\end{center}
where $\E_R(M)$ is the injective envelope of $M$ as an $R$-module.
By  \cite[Theorem 2.11]{wzcc20}, $M_q=\{x\in \T(M[x])|Ix\subseteq M$, for some $I\in \Q \}.$
Obviously, $M$ is a Lucas module if and only if $M_q=M$. A \emph{$\DQ$ ring} $R$ is a ring for which every $R$-module is a Lucas module.
By \cite[Proposition 2.2]{fkxs20}, $\DQ$ rings are exactly rings with small finitistic dimensions equal to $0$.
Recall from \cite{ZDC20} that an submodule $N$ of a  $\Q$-torsion free module $M$ is called a $q$-submodule if $N_q\cap M=N$. If an ideal $I$ of $R$ is a $q$-submodule of $R$, then $I$ is also called a $q$-ideal of $R$. A \emph{maximal $q$-ideal} is an ideal of $R$ which is maximal among the $q$-submodules of $R$. The set of all maximal $q$-ideals is denoted by $q$-$\Max(R)$. $q$-$\Max(R)$ is exactly the set of all maximal non-semi-regular ideals of $R$, and thus is non-empty and a subset of $\Spec(R)$ (see
\cite[Proposition 2.5, Proposition 2.7]{ZDC20}). Recall from \cite{z-q0pvmr} that a ring $R$ is called a \emph{$\WQ$-ring} if every ideal in $\Q$ is a $\GV$-ideal, which is equivalent to $w$- and $q$-operations concide.

An $R$-homomorphism $f:M\rightarrow N$ is called to be  a \emph{$\tau_q$-monomorphism} (resp., \emph{$\tau_q$-epimorphism}, \emph{$\tau_q$-isomorphism}) provided that $f_\m:M_\m\rightarrow N_\m$ is a monomorphism (resp., an epimorphism, an isomorphism) over $R_\m$ for any $\m\in q$-$\Max(R)$. By \cite[Proposition 2.7(5)]{ZDC20}, an $R$-homomorphism $f:M\rightarrow N$ is a $\tau_q$-monomorphism (resp., $\tau_q$-epimorphism, $\tau_q$-isomorphism) if  and only if $\Ker(f)$ is  (resp.,  $\Cok(f)$ is, both $\Ker(f)$ and $\Cok(f)$ are) $\Q$-torsion. A sequence of $R$-modules $A\xrightarrow{f} B\xrightarrow{g} C$ is said to be  \emph{$\tau_q$-exact} provided that $A_\m\rightarrow B_\m\rightarrow C_\m$ is  exact as  $R_\m$-modules for any  $\m\in q$-$\Max(R)$. It is easy to verify a sequence  $A\xrightarrow{f} B\xrightarrow{g} C$ is  $\tau_q$-exact if and only if $(\Im(f)+\Ker(g))/\Im(f)$ and $(\Im(f)+\Ker(g))/\Ker(g)$ are both $\Q$-torsion.

\section{Modules over $\T(R[x])$}

 Let $R$ be a commutative ring.  A finitely generated ideal $J$ of $R$ is said to be a \emph{$\GV$-ideal}, denoted by $J\in \GV(R)$, provided that the natural homomorphism $R\rightarrow \Hom_R(J,R)$ is an isomorphism. Naturally, $\GV(R)\subseteq \Q$. A \emph{$\DW$ ring} $R$ is a ring for which the only $\GV$-ideal of $R$ is $R$ itself. Very recently, the author  et al. \cite[Corollary 3.7]{z-fpd} shows that a ring $R$ is a $\DW$-ring if and only if the small finitistic dimensions of $R$ are at most $1$. For more details on $w$-operations, the reader can consult {\cite[Chapters 6,7]{fk16}}.

 For further study of $w$-operations, Wang and Kim {\cite{fk15}} introduced the \emph{$w$-Nagata ring}, $R\{x\}$, of $R$. It is a localization of the polynomial ring $R[X]$  at the multiplicative closed set
 $S_w=\{f\in R[x]\,|\, c(f)\in \GV(R)\},$
\noindent where $c(f)$ always denotes the content of $f$. Note that $R\{x\}$ is a $\DW$-ring for any ring $R$, i.e., $\fPD(R\{x\})\leq 1.$
Later,  Jara \cite{J15} introduced the $\tau$-Nagata ring $Na(R,\tau)$ for a half-centered hereditary torsion theory $\tau$. Now denote by $\tau_q$ the finite type hereditary torsion theory corresponding to the Gabriel topology generated by $\Q$. Then  following \cite{J15}, there exists a Nagata ring $Na(R,\tau_q):=R[x]_{\sum(\tau_q)}$, where $$\sum(\tau_q):= \{f\in R[x]\,|\, c(f)\in \Q\}.$$ By \cite[Exercise 6.5]{fk16}, $\sum(\tau_q)$ is exactly the set of all regular polynomials in $R[x]$. So we have $$Na(R,\tau_q)=\T(R[x])$$ the quotient ring of $R[x]$. By \cite[Proposition 2.3]{J15}, $\{\m\otimes_R\T(R[x])\mid \m\in q$-$\Max(R)\}$ are exactly the set of all maximal ideals of $\T(R[x])$.

\begin{proposition}\label{naga-lucas}
Let $M$ be a $\T(R[x])$-module. Then $M$ is a Lucas $R$-module.
\end{proposition}

\begin{proof} Let $I$ be a finitely generated semi-regular ideal of $R$ such that $Im=0$ for some $m\in M$. Then $I[x]$ is a regular ideal of $R[x]$ by \cite[Exercise 6.5]{fk16}. So $I\otimes_R\T(R[x])=\T(R[x])$. Hence, we have $0=Im= I\otimes_R\T(R[x])m=\T(R[x])m$ since $M$ is a $\T(R[x])$-module. So $m=0$. It follows that $M$ is $\Q$-torsion-free.

Let $I$ be a finitely generated semi-regular ideal of $R$ such that $In\subseteq M$ for some $n\in \E_R(M)$ where $\E_R(M)$ is the injective envelope of  $M$ as an $R$-module. Since $\T(R[x])$ is a flat $R$-module, we have $\E_R(M)\subseteq \E_{\T(R[x])}(M)$. Since $I\otimes_R\T(R[x])=\T(R[x])$, we have $n\in \T(R[x]) n= I\otimes_R\T(R[x])n\subseteq M\otimes_R\T(R[x])=M$. Consequently, $M$ is a Lucas $R$-module.
\end{proof}

In 2014,  Cahen et al. \cite[Problem 1]{CFFG14} posed the following open question:
\begin{center}
{\bf Problem 1b:} Let $R$ be a total ring of quotients (i.e. $R=\T(R)$). Is $\fPD(R)=0$?
\end{center}
Very recently, Wang et al. \cite{wzcc20} obtained a total ring of quotients $R$ but not a \DQ\ ring, and thus $\fPD(R)>0$ getting a negative answer to Problem 1b. Furthermore, The author et al. \cite[Example 3.10]{z-fpd} constructed an example by idealization to show that there exists a  total ring of quotients $R$ satisfying $\fPD(R)=n$ for any $n\in \mathbb{N}$. However, the next result shows that the the small finitistic dimensions of total quotient rings of all polynomial rings are equal to $0$.

\begin{theorem}\label{naga-fpd}
Let $R$ be a  ring. Then $\T(R[x])$ is a $\DQ$-ring. Consequently, for any ring $R$,  $$\fPD(\T(R[x]))=0.$$
\end{theorem}
\begin{proof} Let $A=\langle f_1,\dots,f_n\rangle_{\sum(\tau_q)}$ be a finitely generated semi-regular ideal of $\T(R[x])$ with $f_1,\dots,f_n\in R[x]$. Set $B=\langle  f_1,\dots,f_n\rangle R[x]$ and $J=c(B)$. 
We claim that the finitely generated ideal $J$ of $R$ is semi-regular. Indeed, let $r$ be an element in $R$ such that $Jr=0$. Then  $rf_i=0$ in $R[x]$ for each $i=1,\dots,n$. Thus  $r\frac{f_i}{1}=0$ in $\T(R[x])$ for each $i=1,\dots,n$. So $rg_if_i=0$ for some regular polynomial $g_i\in R[x]$ ($i=1,\dots,n$). Set $g=\prod\limits_{i=1}^ng_i$. Then $rgf_i=0$ for each $i=1,\dots,n$. Since $A=\langle f_1,\dots,f_n\rangle_{\sum(\tau_q)}$ is a semi-regular ideal of  $\T(R[x])$, we have $rgh=0$ in $R[x]$ for some regular polynomial $h\in R[x]$. Hence $r=0$. It follows that $J$  is semi-regular. Now, let $f\in \langle  f_1,\dots,f_n\rangle R[x]$ such that $c(f)=J$. Then $f\in B\cap \sum(\tau_q)$. So  $A=\T(R[x])$, and thus $\T(R[x])$ is a $\DQ$-ring. By \cite[Proposition 2.2]{fkxs20}, the small finitistic dimension of $\T(R[x])$ is $0$.
\end{proof}

\begin{proposition}\label{q-tor-naga} Let $M$ be a $R$-module.
Then $M$ is a $\Q$-torsion module if and only if $M\otimes_R\T(R[x])=0$.
\end{proposition}
\begin{proof} Suppose $M$ is a $\Q$-torsion module. Let $\sum\limits_{i=1}^nm_i\otimes\frac{f_i}{g_i}$ be an element in $M\otimes_R\T(R[x])$ which  is isomorphic to $ M[x]\otimes_{R[x]}\T(R[x])$ (we identify these two modules). For each $m_i\in M$, there exists an $I_i=\langle a_{i,0},\dots,a_{i,n_i}\rangle \in \Q$ such that $I_im_i=0$. Set $g'_i=a_{i,0}+\cdots+a_{i,0}x^{n_i}$, then $g'_im_i=0$ in $M[x]$.  So
 $$\sum\limits_{i=1}^nm_i\otimes\frac{f_i}{g_i}=\sum\limits_{i=1}^ng'_im_i\otimes\frac{f_i}{g'_ig_i} =0.$$ Hence $M\otimes_R\T(R[x])=0$.

On the other hand, let $m$ be an element in $M$. Then $m\otimes 1=0$ in $M\otimes_R\T(R[x])$. So there exists $r_{j}\in R$ and $\frac{f_j}{g_j}\in \T(R[x])$ $(j=1,\cdots,n)$  such that $1=\sum\limits_{j=1}^nr_j\frac{f_j}{g_j}\in \T(R[x])$  and $mr_j=0$ for each $j=1,\cdots,n$  by \cite[Chapter I, Lemma 6.1]{FS01}. Claim that the ideal $I:=\langle r_1,\dots,r_n\rangle$ is semi-regular. Indeed, note there exists a regular polynomial $s:=s(x)\in R[x]$ such that $\sum\limits_{j=1}^n(r_jsf_j\prod\limits_{i\not=j}g_i)=s\prod\limits_{j=1}^ng_j$. Let $J$ be the ideal generated by the coefficients of regular polynomial $s\prod\limits_{j=1}^ng_j$. Then $J\in \Q$. Note that $J\subseteq I$. So we have  $I\in \Q$. Note that $Im=0$. So  $M$ is a $\Q$-torsion module.
\end{proof}

\begin{proposition}\label{q-exacy-naga} An $R$-sequence $A\rightarrow B\rightarrow C$ is $\tau_q$-exact if and only if  $A\otimes_R\T(R[x])\rightarrow B\otimes_R\T(R[x])\rightarrow C\otimes_R\T(R[x])$ is  $\T(R[x])$-exact.
\end{proposition}

\begin{proof}
 Note that a sequence  $A\xrightarrow{f} B\xrightarrow{g} C$ is  $\tau_q$-exact if and only if $(\Im(f)+\Ker(g))/\Im(f)$ and $(\Im(f)+\Ker(g))/\Ker(g)$ are both $\Q$-torsion, if and only if $\Ker(g)\otimes_R\T(R[x])=(\Im(f)+\Ker(g))\otimes_R\T(R[x])=\Im(f)\otimes_R\T(R[x])$, if and only if $A\otimes_R\T(R[x])\rightarrow B\otimes_R\T(R[x])\rightarrow C\otimes_R\T(R[x])$ is  $\T(R[x])$-exact.
\end{proof}

\begin{corollary}\label{q-mei-naga} An $R$-sequence $f: A\rightarrow B$ is a $\tau_q$-monomorphism $($resp., $\tau_q$-epimorphism, $\tau_q$-isomorphism$)$  if and only if  $f\otimes_R\T(R[x]):  A\otimes_R\T(R[x])\rightarrow B\otimes_R\T(R[x])$ is a monomorphism $($resp., an epimorphism, an isomorphism$)$ over $\T(R[x])$.
\end{corollary}

\section{$\tau_q$-finitely generated modules and $\tau_q$-finitely presented modules}
Recall that an $R$-module $M$ is said to be \emph{$w$-finitely generated} if there exist $w$-exact sequence $F\rightarrow M\rightarrow 0$ with $F$ finitely generated free; and \emph{$w$-finitely presented} if there exist $w$-exact sequence exact sequence $ F_1\rightarrow F_0\rightarrow M\rightarrow 0$ such that $F_0$ and $F_1$ are finitely generated free modules. In this section, we will investigate some $\tau_q$-analogues of classical notions of finitely generated modules and finitely presented modules.

\begin{definition}{ Let $M$ be an $R$-module.

\begin{enumerate}
\item   $M$ is said to be \emph{$\tau_q$-finitely generated} provided that there exists  a $\tau_q$-exact sequence $F\rightarrow M\rightarrow 0$ with $F$ finitely generated free.

\item  $M$ is said to be \emph{$\tau_q$-finitely presented} provided that there exists a $\tau_q$-exact sequence $$ F_1\rightarrow F_0\rightarrow M\rightarrow 0$$ such that $F_0$ and $F_1$ are finitely generated free modules.
\end{enumerate}

 }
\end{definition}

\begin{remark}\label{FPFG}It is easy to verify the following assertions (One can refer to \cite[Section 6.4]{fk16} for the $w$-operation case).
\begin{enumerate}
   \item every $w$-finitely generated (resp., $w$-finitely presented) $R$-module is $\tau_q$-finitely generated (resp., $\tau_q$-finitely presented).

   \item If $R$ is a $\DQ$-ring, then every $\tau_q$-finitely generated (resp., $\tau_q$-finitely presented) module is finitely generated (resp., finitely presented). If $R$ is a $\WQ$-ring, then every $\tau_q$-finitely generated (resp., $\tau_q$-finitely presented) module is $w$-finitely generated (resp., $w$-finitely presented).

   \item An $R$-module $M$ is $\tau_q$-finitely generated if and only if there exists a finitely generated submodule $N$ of $M$ such that $M/N$ is $\Q$-torsion.

   \item An $R$-module $M$ is $\tau_q$-finitely presented if and only if there exist an exact $0\rightarrow T_1\rightarrow N\rightarrow M\rightarrow T_2\rightarrow 0$ with $N$ finitely presented and $T_1, T_2$ $\Q$-torsion.

    \item  A $\tau_q$-finitely generated module $M$ is  $\tau_q$-finitely presented if and only if any $\tau_q$-exact sequence $0\rightarrow A\rightarrow B\rightarrow M\rightarrow 0$ with $B$ $\tau_q$-finitely generated can deduce $A$ is $\tau_q$-finitely generated, if and only if there is a $\tau_q$-exact sequence $0\rightarrow A\rightarrow B\rightarrow M\rightarrow 0$ with $A$ $\tau_q$-finitely generated and $B$  $\tau_q$-finitely presented.

    \item Suppose $f:M\rightarrow N$ is a $\tau_q$-isomorphism. Then $M$ is $\tau_q$-finitely generated (resp., $\tau_q$-finitely presented) if and only if $N$ is $\tau_q$-finitely generated (resp., $\tau_q$-finitely presented).

\end{enumerate}
\end{remark}

\begin{theorem}\label{naga-fg-fp} Let $M$ be an $R$-module. Then the following statements hold.
\begin{enumerate}
   \item $M$ is $\tau_q$-finitely generated if and only if $M\otimes_R\T(R[x])$ is finitely generated over $\T(R[x])$.
   \item $M$ is $\tau_q$-finitely presented if and only if $M\otimes_R\T(R[x])$ is finitely presented over $\T(R[x])$.
\end{enumerate}
\end{theorem}

\begin{proof} $(1)$ Suppose there exists  a $\tau_q$-exact sequence $ F\rightarrow M\rightarrow 0$ with $F$ finitely generated free. Then we have a $\T(R[x])$-exact  sequence  $F\otimes_R\T(R[x])\rightarrow M\otimes_R\T(R[x])\rightarrow 0$. So $M\otimes_R\T(R[x])$ is finitely generated over $\T(R[x])$. On the other hand, suppose $M\otimes_R\T(R[x])$ is finitely generated over $\T(R[x])$. We may assume that $M$ is $\Q$-torsion free by Lemma \ref{q-tor-naga}. Then there exists a finitely generated $R[x]$-submodule $N$ of $M[x]$ such that  $M\otimes_R\T(R[x])=N\otimes_{R[x]}\T(R[x])$. Thus the content $c(N)$ of $N$ is a finitely generated $R$-module such that $N\subseteq c(N)[x]\subseteq M[x]$. So $M\otimes_R\T(R[x])=N\otimes_{R[x]}\T(R[x])\subseteq c(N)[x]\otimes_{R[x]}\T(R[x])\subseteq M[x]\otimes_{R[x]}\T(R[x])=M\otimes_R\T(R[x])$. Hence $M\otimes_R\T(R[x])=c(N)[x]\otimes_{R[x]}\T(R[x])$ is a finitely generated $\T(R[x])$-module.

$(2)$  Suppose there exists  a $\tau_q$-exact sequence $F_1\rightarrow F_0\rightarrow M\rightarrow 0$ such that $F_0$ and $F_1$ are finitely generated free modules.   Then we have a $\T(R[x])$-exact  sequence  $F_1\otimes_R\T(R[x])\rightarrow F_0\otimes_R\T(R[x])\rightarrow M\otimes_R\T(R[x])\rightarrow 0 $. So $M\otimes_R\T(R[x])$ is finitely presented over $\T(R[x])$. On the other hand,  suppose $M\otimes_R\T(R[x])$ is finitely presented $\T(R[x])$-module, then $M$ is $\tau_q$-finitely generated. Consider the $\tau_q$-exact sequence $0\rightarrow L\rightarrow F\xrightarrow{f} M\rightarrow 0$, where $L=\Ker(f)$ and $F$ finitely generated free. Hence  $0\rightarrow L\otimes_R\T(R[x])\rightarrow F\otimes_R\T(R[x])\rightarrow M\otimes_R\T(R[x])\rightarrow 0$ is exact, and so $L\otimes_R\T(R[x])$ is finitely generated. Hence $L$ is $\tau_q$-finitely generated by (1). Consequently, $M$ is $\tau_q$-finitely presented by Remark \ref{FPFG}(5).
\end{proof}

\begin{proposition}\label{lqfg} The following statements hold.

\begin{enumerate}

   \item  The class of $\tau_q$-finitely generated module is closed under extensions.

   \item Any image of the $\tau_q$-finitely generated module is $\tau_q$-finitely generated.
\end{enumerate}
\end{proposition}

\begin{proof}

$(1)$ Let $0\rightarrow M\rightarrow N\rightarrow N/M\rightarrow 0$ be an exact sequence with $M$ and $N/M$ $\tau_q$-finitely generated. There exist $M_1=\langle m_1,\dots,m_s\rangle$ and $N'/M=\langle n_1+M,\dots,n_t+M\rangle$ be submodules of $M$ and $N/M$ respectively such that $M/M_1$ and $N/N'$ are all $\Q$-torsion. Set $N_1=\langle m_1,\dots,m_s,n_1,\dots,n_t\rangle$. And then $ N_1/M\cong N'/M$. Thus there is a commutative diagram $$\xymatrix@R=20pt@C=20pt{
 0 \ar[r]^{} & M_1\ar[d]_{i_1}\ar[r]^{} & N_1 \ar[d]^{i_2}\ar[r]^{} & N_1/M\ar[d]^{i_3}\ar[r]^{} &  0\\
 0 \ar[r]^{} &  M\ar[r]^{g} &N \ar[r]^{} & N/M \ar[r]^{} &  0.}$$
Since $\Q$-torsion is closed under extensions, $\coker(i_2)$ is $\Q$-torsion. Thus $N$ is $\tau_q$-finitely generated.

$(2)$ Let $0\rightarrow M\rightarrow N\rightarrow N/M\rightarrow 0$ be an exact sequence with $N$ $\tau_q$-finitely generated. Then there exist $N_1=\langle n_1,\dots,n_k\rangle$ such that $N/N_1$ is $\Q$-torsion. Thus $(N_1+M)/M$ is a finitely generated submodule of $N/M$ such that the quotient $N/(N_1+M)$ is a quotient module of $N/N_1$. Thus $N/(N_1+M)$ is $\Q$-torsion and $N/M$ is $\tau_q$-finitely generated.
\end{proof}

\begin{proposition}\label{ex-fp} Let $0\rightarrow M\rightarrow N\rightarrow L\rightarrow 0$ be an $\tau_q$-exact sequence and $N$  $\tau_q$-finitely presented, then $L$ is $\tau_q$-finitely presented if and only if  $M$ is $\tau_q$-finitely generated.
\end{proposition}

\begin{proof} The equivalence can easily be deduced by Remark \ref{FPFG}(5), so we omit it.
 \end{proof}

\begin{corollary}\label{cap-fp} Suppose $N_1$ and $N_2$ $\tau_q$-finitely presented submodule of an $R$-module $M$, then $N_1+N_2$ is $\tau_q$-finitely presented if and only if  $N_1\cap N_2$ is $\tau_q$-finitely generated.
\end{corollary}

\begin{proof} Suppose $N_1$ and $N_2$ are $\tau_q$-finitely presented. Then  $ N_1\oplus N_2$ is also $\tau_q$-finitely presented. The consequence follows from  by Proposition \ref{ex-fp} and  the natural exact sequence $0\rightarrow N_1\cap N_2\rightarrow N_1\oplus N_2\rightarrow N_1+N_2\ra 0.$
\end{proof}

\section{$\tau_q$-flat modules}

Following \cite{fk15}, an $R$-module $M$ is said to be \emph{$w$-flat} provided that the induced sequence $1\otimes_R f:M\otimes_R A\rightarrow M\otimes_R B$ is a $w$-monomorphism for any $w$-monomorphism $f: A\rightarrow B$.  In this section, we study the following $\tau_q$-analogue of flat modules.

\begin{definition}{An $R$-module $M$ is said to be a \emph{$\tau_q$-flat} module provided that, for any $\tau_q$-monomorphism  $f: A\rightarrow B$,  $1_M\otimes f:M\otimes_RA\rightarrow M\otimes_R B$ is a $\tau_q$-monomorphism.}
\end{definition}

\begin{lemma}Let $T$ be a $\Q$-torsion $R$-module. Then $\Tor_1^n(T,M)$ is $\Q$-torsion for any $R$-module $M$ and any $n\geq 0$.
\end{lemma}
\begin{proof}
Let $\m$ be a maximal $q$-ideal of $R$. Then we have
$$\Tor_n^R(T,M)_\m\cong \Tor_n^{R_\m}(T_\m,M_\m).$$ Since  $T$ be a $\Q$-torsion $R$-module, $T_\m=0$ by \cite[Proposition 2.7]{ZDC20}. So $\Tor_n^R(T,M)_\m=0$. It follows from \cite[Proposition 2.7]{ZDC20} again that $\Tor_n^R(T,M)$ is $\Q$-torsion.
\end{proof}

It is well-known that an $R$-module $M$  is $w$-flat, if and only if $M_\m$ is a flat $R_\m$-module for any $\m\in w$-$\Max(R)$, if and only if $\Tor_1^R(N,M)$ is $\GV$-torsion for any $R$-module $N$, if and only if the natural homomorphism $J\otimes_R M\rightarrow JM$ is a $w$-isomorphism for any ideal $J$, if and only if $M\otimes_RR\{x\}$ is a flat $R\{x\}$-module.  (see \cite[Theorem 3.3]{fk15} and \cite[Lemma 4.1]{WQ15}). Now, we give some characterizations of $\tau_q$-flat modules.

\begin{theorem}\label{w-coh-c-c}
The following statements are equivalent for an $R$-module $M$.
\begin{enumerate}
    \item $M$ is $\tau_q$-flat;
    \item   for any monomorphism  $f: A\rightarrow B$,  $1_M\otimes f:M\otimes_RA\rightarrow M\otimes_R B$ is a $\tau_q$-monomorphism
    \item  For any $N$, $\Tor^R_1(M,N)$ is $\Q$-torsion;
      \item  For any $N$ and $n\geq 1$, $\Tor^R_n(M,N)$ is $\Q$-torsion;
      \item  For any ideal $I$, the natural homomorphism $M\otimes_RI\rightarrow MI$ is a $\tau_q$-isomorphism;
       \item  For any  finitely generated $(\tau_q$-finitely generated$)$ ideal $I$, the natural homomorphism $M\otimes_RI\rightarrow MI$ is a $\tau_q$-isomorphism;
          \item  For any finitely generated $(\tau_q$-finitely generated$)$ ideal $I$,   $\Tor^R_1(R/I,M)$ is $\Q$-torsion;
          \item $M_\m$ is a flat $R_\m$-module for any $\m\in q$-$\Max(R)$;
        \item $M\otimes_R\T(R[x])$ is a flat  $\T(R[x])$-module.
\end{enumerate}
\end{theorem}

\begin{proof}

$(1)\Rightarrow (2)$, $(4)\Rightarrow (3)\Rightarrow (7)$ Trivial.

$(2)\Rightarrow (8)$ Let $K=I_{\m}$ be an ideal of $R_\m$ with $I$ an ideal of $R$. By Localizing at $\m$, we consider the following exact sequence $$0\rightarrow \Ker(1_M\otimes i)_{\m}\rightarrow M_{\m}\otimes_{R_{\m}}I_{\m}\rightarrow M_{\m}\otimes_{R_{\m}} R_{\m}.$$
Since $\Ker(1_M\otimes f)$ is $\Q$-torsion, $ \Ker(1_M\otimes i)_{\m}=0$. So $M_{\m}$ is a flat $R_\m$-module.

$(8)\Rightarrow (4)$ It follows from $\Tor^R_n(M,N)_\m\cong \Tor^{R_\m}_n(M_\m,N_\m)=0$.

$(5)\Leftrightarrow (6)\Leftrightarrow (8)$  Let $\m$ be a maximal $q$-ideal of $R$.  The equivalences follow from \cite[Theorem 2.5.6]{fk16}  by Localizing at $\m$.

$(7)\Rightarrow (8)$  Let $K=I_{\m}$ be a finitely generated ideal of $R_\m$ with $I$ a finitely generated ideal of $R$. Then  $\Tor^{R_\m}_1(R_\m/K,M_\m)=\Tor^R_1(R/I,M)_\m=0$. So $M_\m$ is flat over $R_\m$ by \cite[Theorem 3.4.6]{fk16}.

$(8)\Rightarrow (1)$ Let $f:A\rightarrow B$ be a $\tau_q$-monomorphism.  Then $f_\m:A_\m\rightarrow B_\m$ is a monomorphism. By $(8)$, $f_\m\otimes_{R_{\m}} M_\m:A_\m\otimes_{R_{\m}} M_\m\rightarrow B_\m\otimes_{R_{\m}} M_\m$ is a monomorphism. Hence $1_M\otimes f:M\otimes_RA\rightarrow M\otimes_R B$ is a $\tau_q$-monomorphism.

$(1)\Rightarrow (9)$ Let $N$ be an $\T(R[x])$-module. Since $\T(R[x])$ is a flat $R$-module, we have $\Tor_1^{\T(R[x])}(M\otimes_R\T(R[x]),N)\cong \Tor_1^R(M,N)$ which is $\Q$-torsion. By Proposition \ref{naga-lucas}, $\Tor_1^{\T(R[x])}(M\otimes_R\T(R[x]),N)$ is  a Lucas module. Hence $M\otimes_R\T(R[x])$ is flat over $\T(R[x])$.

$(9)\Rightarrow (3)$ Let $N$ be an $R$-module. Note $$\T(R[x])\otimes_R\Tor^R_1(M,N)\cong \Tor^{\T(R[x])}_1(\T(R[x])\otimes_RM,\T(R[x])\otimes_RN)=0.$$ So $\Tor^R_1(M,N)$ is $\Q$-torsion by Proposition \ref{q-tor-naga}.
\end{proof}

By Theorem \ref{w-coh-c-c}, we can easily obtain the following result.
\begin{corollary} Suppose $f:M\rightarrow N$ is a $\tau_q$-isomorphism. Then $M$ is $\tau_q$-flat if and only if  $N$ is $\tau_q$-flat.
\end{corollary}

\begin{proposition} A ring $R$ is a $\DQ$-ring if and only if any $\tau_q$-flat module is flat.
\end{proposition}
\begin{proof} Obviously, if $R$ is a $\DQ$-ring, then  any $\tau_q$-flat module is flat. On the other hand, let $I$ be a finitely generated semi-regular ideal of $R$. Then $R/I$ is $\Q$-torsion, and hence $\tau_q$-flat. So  $R/I$ is a flat module. It follows that $\Tor_1^R(R/I,R/I)\cong I/I^2=0$ (see \cite[Exercise 3.20]{fk16}). Then $I$ is finitely generated idempotent ideal, and hence $I=\langle e\rangle$ for some idempotent element $e\in I$ by \cite[I, Proposition 1.10]{FS01}. Since  $I$ is semi-regular ideal,  $1-e\in (0:_RI)=0$. So $e=1$ and $I=R$.
\end{proof}

\begin{proposition}  A ring $R$ is a $\WQ$-ring if and only if any $\tau_q$-flat module is $w$-flat.
\end{proposition}
\begin{proof} Obviously, if $R$ is a $\WQ$-ring, then  any $\tau_q$-flat module is $w$-flat. On the other hand, let $I$ be a finitely generated semi-regular ideal of $R$. Then $R/I$ is $\Q$-torsion, and hence $\tau_q$-flat. So  $R/I$ is a $w$-flat module. It follows that $\Tor_1^R(R/I,R/I)\cong I/I^2$ is $\GV$-torsion. Thus $ I\subseteq (I^2)_w$ for any $I\in \Q$. By \cite[Theorem 4.2]{z-q0pvmr}, $R$ is a $\WQ$-ring.
 \end{proof}

In the rest of this section, we explore  the rings over which all $R$-modules are $\tau_q$-flat.

\begin{definition}{A ring $R$ is said to be a \emph{$\tau_q$-VN regular ring} $($short for  $\tau_q$-von Neumann regular ring$)$ provided that all $R$-modules are $\tau_q$-flat.}
\end{definition}

First we recall  some notions on Serre's conjecture ring $R\langle x\rangle$  of a ring $R$ (see \cite{S55}). Let $S_1=\{f\in R[x]\mid f\ \mbox{is a monoic polynomial in}\ R[x]\}.$  Then the Serre's conjecture ring defined to be  $R\langle x\rangle:=R[x]_{S_1}$. Note that  $R\langle x\rangle$ is a faithful flat $R$-module (see \cite[Lemma 3.5]{fk15}).

\begin{lemma}\label{naga-vn}{Suppose $R$ is a von Neumann regular ring. Then $R\langle x\rangle$ is also a von Neumann regular ring.}
\end{lemma}
\begin{proof} Let  $R$ be a von Neumann regular ring. Then $R$ is a reduced ring with Krull dimension equal to $0$. Since $\Nil(R\langle x\rangle)=\Nil(R)\langle x\rangle$, we have $R$ is reduced if and only if $R\langle x\rangle$ is reduced. By \cite[Theorem 5.5.9]{fk16}  $R\langle x\rangle$ is also  a reduced ring with Krull dimension equal to $0$. Consequently, $R\langle x\rangle$ is also a von Neumann regular ring by \cite[the Remark in Page 5]{H88}.
\end{proof}

Following \cite[Theorem 4.4]{fk14}, a ring $R$ is von Neumann regular ring if and only if any $R$-module is $w$-flat. So von Neumann regular rings are exactly ``$w$-von Neumann regular ring''. Now, we can give the characterizations of  $\tau_q$-VN regular rings.
\begin{theorem}\label{char-tvn}
The following assertions are equivalent for a ring $R$.
\begin{enumerate}
    \item $R$ is a $\tau_q$-VN regular ring;
     \item for any finitely generated ideal $K$ of $R$, there exists an $I\in\Q$ such that $IK=K^2$;
    \item $R_{\m}$ is a von Neumann regular ring for any $\m\in q$-$\Max(R)$;
     \item   $\T(R[x])$ is a von Neumann regular ring;
        \item    $R$ is a reduced ring and $\Min(R)$ is compact.
\end{enumerate}
\end{theorem}

\begin{proof} $(1)\Rightarrow (2)$  Let $K$ be a finitely generated ideal of $R$. Then $R/K$ is $\tau_q$-flat by (1). So the finitely generated $R$-module $\Tor_1^R(R/K,R/K)=K/K^2$ is $\Q$-torsion. Hence  there exists an $I'\in\Q$ such that $I'K\subseteq  K^2$. Let $I=I'+K$, then $I\in\Q$ and $IK=K^2$.

$(2)\Rightarrow (3)$  Let $\m$ be in $q$-$\Max(R)$ and $K_\m$ a finitely generated ideal of $R_\m$ with $K$ a finitely generated ideal of $R$. Then  there exists an $I\in\Q$ such that $IK=K^2$. So $I_\m K_\m=K_\m^2$. Note that $I_\m=R_\m$ by \cite[Proposition 2.7]{ZDC20}. Hence $K_\m=K_\m^2$. By
\cite[Proposition 1.10]{FS01}, $K_\m$ is generated by an idempotent.  It follows from \cite[Theorem 3.6.13]{fk16} that $R_{\m}$ is a von Neumann regular ring.


$(3)\Rightarrow (4)$ Let $\p$ be a maximal ideal of $\T(R[x])$, then $\p=\m\otimes_R\T(R[x])$ for some $\m\in q$-$\Max(R)$
 by \cite[Proposition 2.3]{J15}. Let $\sum(\tau_q)$ be the set of all polynomials with contents in $\Q$. Then it is easy to verify $\sum(\tau_q)\subseteq R[X]-\m[X]$, so we have $$\T(R[x])_\p=\T(R[x])_{\m\otimes_R\T(R[x])}\cong (R[x]_{\sum(\tau_q)})_{\m[x]_{\sum(\tau_q)}}\cong R[x]_{\m[x]} \cong R_{\m}\langle x\rangle, $$ by \cite[Lemma 3.5]{fk15}. By Lemma \ref{naga-vn}, $R_{\m}\langle x\rangle$ is also a von Neumann regular ring. Then $\T(R[x])_\p$ is a von Neumann regular ring for any maximal ideal $\p$ of $\T(R[x])$. So $\T(R[x])$ is also a von Neumann regular ring.

$(4)\Rightarrow (1)$ Let $M$ be an $R$-module. Then $M\otimes_R\T(R[x])$ is a $\T(R[x])$-module, and so is flat by $(3)$. Hence $M$ is a $\tau_q$-flat $R$-module by Theorem \ref{w-coh-c-c}.

$(2)+(4)\Rightarrow (5)$ First we claim $R$ is reduced.  On contrary, suppose there exists a nonzero element $a$ in $R$ such that $a^n=0$ for some smallest integer $n\geq 2$. Then there exists an $I\in \Q$ such that $Ia=\langle a^2\rangle$ by Theorem \ref{char-tvn}. Hence $(0:_Ra^2)=(0:_RIa)=((0:_RI):_Ra)=(0:_Ra)$. Since $a^{n-2}\in (0:_Ra^2)=(0:_Ra)$, we have $a^{n-1}=0$ which is a contradiction. So $a=0$. Then, by \cite[Corollary 4.6]{H88}, $\Min(R)$ is compact.

$(5)\Rightarrow (4)$ Follows by \cite[Corollary 4.6]{H88}.
\end{proof}

Following Theorem \ref{char-tvn}, we certainly have the following result.
\begin{corollary}\label{in-vn} Let $R$ be a von Neumann regular ring or an integral domain. Then $R$ is a  $\tau_q$-VN regular ring.
\end{corollary}

Certainly,  $\Q$-torsion modules  and $w$-flat modules are $\tau_q$-flat. However,  $\tau_q$-flat modules need not be  $w$-flat. Indeed, let $R$ be an integral domain not a field. Then every $R$-module is $\tau_q$-flat by Corollary \ref{in-vn}. However, there exists an $R$-module $M$ which is not $w$-flat by
 \cite[Theorem 4.4]{fk14}.

\section{$\tau_q$-coherent rings}
Following \cite{fk12}, an $R$-module $M$ is said to be a \emph{$w$-coherent module}, if $M$ is  $w$-finitely presented and any $w$-finitely generated submodule of $M$ is $w$-finitely presented. And a ring  $R$ is called a \emph{$w$-coherent ring} provided that $R$ is $w$-coherent as an $R$-module.  Now, we introduce and study the $\tau_q$-analogues of coherent modules and coherent rings.

\begin{definition}{An $R$-module $M$ is said to be a \emph{$\tau_q$-coherent module} provided that every $\tau_q$-finitely generated submodule of $M$ is $\tau_q$-finitely presented.
A ring $R$ is a \emph{$\tau_q$-coherent ring} provided that $R$ is $\tau_q$-coherent as an $R$-module.}
\end{definition}

\begin{remark}\label{re-tq-coh}It is easy to verify the following assertions (One can refer to \cite[Section 6.9.2]{fk16} for the $w$-operation case).
\begin{enumerate}
   \item Suppose $f:M\rightarrow N$ is a $\tau_q$-isomorphism of $R$-modules. Then $M$ is $\tau_q$-coherent  if and only if $N$ is  $\tau_q$-coherent.
  \item All $w$-coherent rings and integral domains are $\tau_q$-coherent.

   \item If $R$ is a $\DQ$-ring, then $\tau_q$-coherent modules are exactly coherent modules.  If $R$ is a $\WQ$-ring, then $\tau_q$-coherent modules are exactly $w$-coherent modules.

   \item A ring $R$ is a  $\tau_q$-coherent ring if and only if every finitely generated ideal of $R$ is $\tau_q$-finitely presented, if and only if for any $n\geq 1$
       every finitely generated submodule of $R^n$ is $\tau_q$-finitely presented, if and only if for any $n\geq 1$
       every $\tau_q$-finitely generated submodule of $R^n$ is $\tau_q$-finitely presented, if and only if every $\tau_q$-finitely presented $R$-module is $\tau_q$-coherent.

\end{enumerate}
\end{remark}

Recently, the authors in \cite{ZDC20} introduced the notion of  $\tau_q$-Noetherian rings which are commutative rings with  all ideals  $\tau_q$-finitely generated.

\begin{proposition}\label{tau-N-c} Every $\tau_q$-Noetherian ring is $\tau_q$-coherent.
\end{proposition}

\begin{proof}  Let $R$ be a $\tau_q$-Noetherian ring and $I$ an $n$-generated ideal of $R$, then there is an exact sequence $0\rightarrow K\rightarrow R^n\rightarrow I\rightarrow 0$. Since $R$ is $\tau_q$-Noetherian, we can induced by $n$ that $K$ is $\tau_q$-finitely generated. By Proposition \ref{ex-fp}, $I$ is $\tau_q$-finitely presented. So $R$ is also $\tau_q$-coherent.
\end{proof}

Note that the converse of Propostion \ref{tau-N-c} does not always hold.
\begin{example}
Let $R$ be a von Neumann regular which is not semi-simple. The $R$ is coherent but not Noetherian. So it is also $\tau_q$-coherent. Let $I$ be a finitely generated semiregular ideal of $R$, then $I$ is generated by a idempotent unit $e$ and thus  $I=R$. So $R$ is  a $\DQ$-ring. Thus every $\tau_q$-finitely generated ideal of $R$ is finitely generated. Consequently,  $R$ is not a $\tau_q$-Noetherian ring.
\end{example}

Note that $\tau_q$-coherent rings are not always $w$-coherent. Indeed, every  non-$w$-coherent  domain is $\tau_q$-coherent. Some non-integral domain examples  are provided by the idealization $R(+)M$, where $M$ is an $R$-module (see \cite{H88}). We recall this construction. Let  $R(+)M=R\oplus M$ as an $R$-module, and define
\begin{enumerate}
    \item ($r,m$)+($s,n$)=($r+s,m+n$).
    \item  ($r,m$)($s,n$)=($rs,sm+rn$).
\end{enumerate}
Under these definitions, $R(+)M$ becomes a commutative ring with identity.  Now we give a $\tau_q$-coherent ring with zero divisors which is not  $w$-coherent.
\begin{example}\label{tau-N-c-exam}  Let $D$ be a discrete valuation domain which is not a field, $Q$ the quotient field of $D$. Then $Q$ is not finitely generated as a $D$-module. Set $R=D(+)Q$. Then $\Nil(R)=0(+)Q$ and every non-nilpotent element is regular.

First, we will  show $R$ is not $w$-coherent. Let $R(0,1)$ be the $R$-ideal generated by $(0,1)$. Then there is a natural exact sequence $0\rightarrow 0(+)Q\rightarrow R\rightarrow R(0,1)\rightarrow 0$. By \cite[Lemma 2.2]{B03},  $0(+)Q$ is not finitely generated, thus is not of $w$-finitely generated since  $0(+)Q=\Nil(R)$ is a  $w$-ideal. So $R(0,1)$ is not $w$-finitely presented. Hence $R$ is not a $w$-coherent ring.

Now we claim that $R$ is a $\tau_q$-Noetherian ring. Let $I$ be an ideal of $R$. Then, by \cite[Corollary 3.4]{DW09}, $I$ is either of the form $K(+)Q$ where $K$ is an ideal of $D$, or $0(+)B$ where $B$ is a $D$-submodule of $Q$. If $I=K(+)Q$, then $I$ is certainly finitely generated as $K$ is finitely generated. Suppose $I=0(+)B$ where $B$ is a nonzero submodule of $Q$. If $B$ is an ideal of $D$, then $I=0(+)B$ is also  finitely generated.  Otherwise, since $R$ is a valuation domain, we can assume $B$ is a $D$-module strictly containing $D$.  We will show $0(+)B/0(+)D$ is $\Q$-torsion. Let $\frac{t}{s}$ be an element in $B-D$. Then $J=\langle s\rangle(+)Q$ is a finitely generated regular ideal of $R$ satisfying  $J(0,t/s)\subseteq \langle (0,t)\rangle \subseteq 0(+)D$. Thus $0(+)B/0(+)D$ is $\Q$-torsion. So $I$ is $\tau_q$-finitely generated, and hence $R$ is a $\tau_q$-Noetherian ring. By Proposition \ref{tau-N-c}, $R$ is also  $\tau_q$-coherent.
\end{example}

We also provide a non-$\tau_q$-coherent ring by the idealization.
\begin{example} Let $k$ be a field and $V=\prod\limits_{i=1}^\infty k$ the countably infinite  product of copies of $k$. Set $R=k(+)V$.   First we claim that $R$ is not a coherent ring. Indeed, let $v$ be the element in $V$ with all components $1$. Then we have an $R$-exact sequence $0\rightarrow 0(+)V\rightarrow R\rightarrow R(0,v)\rightarrow 0$. Since $0(+)V$ is not finitely generated, the principal ideal $R(0,v)$ is not finitely presented. So $R$ is not a coherent ring. Note that  the Krull dimension of $R$ is zero, and thus $R$ is a $\DQ$-ring  by \cite[Proposition 3.14]{wzcc20} and \cite[Proposition2.2]{fkxs20}.  So $R$ is also not a $\tau_q$-coherent ring by Remark \ref{re-tq-coh}(3).
\end{example}

The rest of this section will give several characterizations of $\tau_q$-coherent rings.
\begin{theorem}\label{w-coh-c-re}
The following statements are equivalent for a ring $R$.
\begin{enumerate}
    \item $R$ is $\tau_q$-coherent;
        \item Any intersection of two finitely generated ideals of $R$ is  $\tau_q$-finitely generated and  $(0:_Rb)=\{r\in R|rb=0\}$  is $\tau_q$-finitely generated for every $b\in R$;
      \item  $(I:_R b)=\{r\in R|rb\in I\}$ is $\tau_q$-finitely generated for every $b\in R$ and finitely generated ideal $I$.
\end{enumerate}
\end{theorem}

\begin{proof}

$(1)\Rightarrow (2)$ Let $N_1$ and $N_2$ be finitely generated ideals of $R$. By (1), $N_1$ and $N_2$ are $\tau_q$-finitely presented. Thus $N_1\oplus N_2$ is $\tau_q$-finitely presented.  Since $N_1+N_2$ is finitely generated ideals, $N_1\cap N_2$ is  $\tau_q$-finitely generated by Corollary \ref{cap-fp}. Consider the exact sequence $0\rightarrow (0:_Rb)\rightarrow R\rightarrow Rb\rightarrow 0$. Since $Rb$ is $\tau_q$-finitely presented, $(0:_Rb)$ is  $\tau_q$-finitely generated by Proposition \ref{ex-fp}.

$(2)\Rightarrow (1)$  Let $N$ be a finitely generated ideal generated by $n$ elements. We claim that $N$ is $\tau_q$-finitely presented. If $n=1$,  we can conclude  by the exact sequence  $0\rightarrow (0:_Rb)\rightarrow R\rightarrow Rb\rightarrow 0$. Suppose the claim holds for $n=k$. Now suppose $n=k+1$, then $N=L+Ra$ where $L$ is generated by $k$ element. Consider the exact sequence $0\rightarrow L\cap Ra\rightarrow L\oplus Ra\rightarrow L+Ra\rightarrow 0$, we obtain $ L\oplus Ra$ is $\tau_q$-finitely presented by induction. By (2), $L\cap Ra$ is $\tau_q$-finitely generated, thus $N=L+Ra$ is $\tau_q$-finitely presented by Corollary \ref{cap-fp}.

$(1)\Rightarrow (3)$ Let $I$ be a finitely generated ideal and $b$ an element in $R$. Consider the exact sequence commutative diagram
 $$\xymatrix@R=20pt@C=15pt{
 0 \ar[r]^{}  & (I:_R b)\ar[d]_{}\ar[r]^{} &R\ar[r]^{}\ar[d]_{} & (Rb+I)/I\ar[d]^{}\ar[r]^{} &  0\\
   0 \ar[r]^{} &I \ar[r]^{} &Rb+I \ar[r]^{} & (Rb+I)/I\ar[r]^{} &  0.}$$
Since $R$ is $\tau_q$-coherent and $I$ is finitely generated, we have $Rb+I$ is $\tau_q$-finitely presented. By Proposition \ref{ex-fp}, $(Rb+I)/I$ is $\tau_q$-finitely presented, thus $(I:_R b)$ is $\tau_q$-finitely generated.

$(3)\Rightarrow (1)$   Let $N$ be a finitely generated ideal generated by $n$ elements. We will show that $N$ is $\tau_q$-finitely presented by induction. If $n=1$,  we can conclude  by the exact sequence  $0\rightarrow (0:_Rb)\rightarrow R\rightarrow Rb\rightarrow 0$. Suppose the claim holds for $n=k$. Now suppose $n=k+1$, then $N=L+Ra$ where $L$ is generated by $k$ element. Consider the exact sequence $0 \ra (L:_R a)\ra R\ra (Ra+L)/L\ra 0$. By Proposition \ref{ex-fp}, $(Ra+L)/L= N/L$ is $\tau_q$-finitely presented by induction. Then  $N$ is $\tau_q$-finitely presented follows from the exact sequence $0 \rightarrow L \rightarrow N \rightarrow N/L\rightarrow 0$.
\end{proof}

We say that  an $R$-module $M$ is \emph{$\tau_q$-Mittag-Leffler} with respect to flat modules if the canonical map
\begin{center}
$\phi_{M}: M\bigotimes_R\prodi F_i\rightarrow \prodi  (M\bigotimes_R F_i)$
\end{center}
is a $\tau_q$-monomorphism for any family $\{F_i\}_{i\in I}$ of flat modules. Similar with \cite[Theorem 2.9]{zxl20}, one can obtain the result (we omit the proof since it is similar with that of  \cite[Theorem 2.9]{zxl20}).

\begin{lemma}\label{w-iso}
The class of $\tau_q$-Mittag-Leffler modules with respect to flat modules is closed under $\tau_q$-isomorphisms. Consequently, every $\tau_q$-finitely presented module is  $\tau_q$-Mittag-Leffler with respect to flat modules.
\end{lemma}

Recall from \cite{WQ20}, an $R$-module $M$ is called $\tau_q$-divisible provided $IM=M$ for any $I\in \Q$, and $\prod$-flat if $M^{\Gamma}$ is flat for any index set $\Gamma$.
Now, we give the Chase theorem for $\tau_q$-coherent rings.

\begin{theorem}\label{w-coh-chase}
Let $R$ be a ring. Then the following assertions are equivalent:
\begin{enumerate}
    \item $R$ is $\tau_q$-coherent;
      \item  the direct product of any family of $\tau_q$-divisible flat $R$-module is flat;
      \item  any $\tau_q$-divisible flat $R$-module is $\prod$-flat;
       \item  $\T(R[x])$ is a $\prod$-flat $R$-module;
       \item any direct product of flat modules is $\tau_q$-flat;
       \item any direct product of projective modules is $\tau_q$-flat;
       \item any direct product of $R$ is $\tau_q$-flat.
\end{enumerate}
\end{theorem}
\begin{proof} The equivalence $(1)\Leftrightarrow (2)\Leftrightarrow (3)\Leftrightarrow (4)$ follows from \cite[Theorem 2.10]{WQ20}.

$(5)\Rightarrow (6)\Rightarrow (7)$ Trivial.

$(1)\Rightarrow (5)$ Let $\{F_i\mid i\in \Lambda\}$ be a family of flat $R$-modules and $I$ a finitely generated ideal of $R$. Then $I$ is $\tau_q$-finitely presented.  Consider the following commutative diagram with rows exact:
$$\xymatrix{
0\ar[r]& \Tor_1^R(R/I,\prodi F_i)\ar[d]_{}\ar[r]^{} & I\otimes_R \prodi F_i \ar[d]^{\theta}\ar[r]^{} & R\otimes_R  \prodi F_i \ar[d]^{}\\
   0 \ar@{=}[r] & \prodi\Tor_1^R(R/I, F_i)\ar[r]^{} &\prodi (I\otimes_R F_i ) \ar[r]^{} & \prodi (R\otimes_R F_i),\\}$$
Note $\Tor_1^R(R/I,\prodi F_i)\subseteq \Ker(\theta)$. By Lemma \ref{w-iso}, $\Ker(\theta)$  is $\Q$-torsion, and so is $\Tor_1^R(R/I,\prodi F_i)$. Hence  $\prodi F_i$ is $\tau_q$-flat by Theorem \ref{w-coh-c-c}.

$(7)\Rightarrow (1)$ Let $I$ be a finitely generated ideal of $R$.  Consider the  following commutative diagram with rows exact:
$$\xymatrix{
&I\otimes_R \prodi R\ar[d]_{g}\ar[r]^{f} & R\otimes_R  \prodi R \ar[d]_{}^{\cong}\ar[r]^{} & R/I\otimes_R  \prodi R \ar[d]_{}^{\cong}\ar[r]^{} &  0\\
   0 \ar[r]^{} & \prodi (I\otimes_RR)\ar[r]^{} &\prod_{i\in I} (R\otimes_RR ) \ar[r]^{} & \prodi (R/I\otimes_RR ) \ar[r]^{} &  0.\\}$$
Since $\prodi R$ is a $\tau_q$-flat module, then $f$ is a $\tau_q$-monomorphism. So $g$ is also a $\tau_q$-monomorphism.

Let $0\rightarrow L\rightarrow F\rightarrow I\rightarrow 0$ be an exact sequence with $F$ finitely generated free.
 Consider the  following commutative diagram with rows exact:
$$\xymatrix{
&L\otimes_R \prodi R\ar[d]_{h}\ar[r]^{} & F\otimes_R  \prodi R \ar[d]_{}^{\cong}\ar[r]^{} & I\otimes_R  \prodi R \ar[d]_{}^{g}\ar[r]^{} &  0\\
   0 \ar[r]^{} & \prodi (L\otimes_RR)\ar[r]^{} &\prod_{i\in I} (F\otimes_RR ) \ar[r]^{} & \prodi (I\otimes_RR ) \ar[r]^{} &  0.\\}$$
Since $g$ is a  $\tau_q$-monomorphism, $h$ is a $\tau_q$-epimorphism. Set $\Lambda$  equal to the cardinal of $L$. We will show $L$ is $\tau_q$-finitely generated. Indeed, consider the following exact sequence
$$\xymatrix{
L\otimes_R  R^L \ar[rr]^{h}\ar@{->>}[rd] &&L^L \ar[r]^{} & T\ar[r]^{} &  0\\
    &\Im h \ar@{^{(}->}[ru] &&  &   \\}$$
with $T$ a $\Q$-torsion module.
Let $x=(m)_{m\in L}\in L^L$. Then there is an ideal $J=\langle u_1,\dots,u_k\rangle\in \Q$ such that $Jx\subseteq \Im \phi_L$.
Subsequently, there exist $m_j\in L, r_{j,i}\in R, i\in L, j=1,\cdots,n$
such that for each $t=1,\dots,k$, we have
$$u_tx=h(\sum_{j=1}^n m^t_j\otimes (r^t_{j,i})_{i\in L})=(\sum_{j=1}^n m^t_j r^t_{j,i})_{i\in L}.$$
Set $U=\langle m^t_j\mid j=1,\dots,n; t=1,\dots,k\rangle $ be the finitely generated submodule of $L$. Now, for any $m\in L$, $Jm\subseteq \langle \sum_{j=1}^n m^t_j r^t_{j,m}\mid t=1,\dots,k\rangle\subseteq  U$, thus the embedding map $U\hookrightarrow L$ is a $\tau_q$-isomorphism and so $L$ is  $\tau_q$-finitely generated. Consequently, $I$ is $\tau_q$-finitely presented.
\end{proof}

Following Theorem \ref{w-coh-chase}, all $\tau_q$-VN regular rings are  $\tau_q$-coherent. However, the converse does not hold in general. Indeed, let $R$ be the ring in Example \ref{tau-N-c-exam}. Then $R$ is a non-reduced $\tau_q$-coherent ring. So $R$ is not a $\tau_q$-VN regular ring by Theorem \ref{char-tvn}.

\begin{acknowledgement}\quad\\
The first author was supported by National Natural Science Foundation of China (No. 12061001). The second author was supported by National Natural Science Foundation of China (No. 12201361).
\end{acknowledgement}

\end{document}